\documentclass{amsart}

\usepackage{amssymb}

\usepackage{lmodern}
\usepackage[T1]{fontenc}

\usepackage[all]{xy}

\usepackage{fullpage}
\usepackage{paralist}

\usepackage{color}
\usepackage[normalem]{ulem}

\usepackage[colorlinks,pagebackref=true]{hyperref}
\usepackage[alphabetic, abbrev, lite, backrefs]{amsrefs}

\makeatletter
\def\@citecolor{blue}
\makeatother

\theoremstyle{plain}
\newtheorem{thm}{Theorem}[section]
\newtheorem{lemma}[thm]{Lemma}
\newtheorem{cor}[thm]{Corollary}
\newtheorem{propn}[thm]{Proposition}

\theoremstyle{definition}
\newtheorem{defn}[thm]{Definition}
\newtheorem{example}[thm]{Example}
\newtheorem{question}[thm]{Question}

\newtheorem{constr}[thm]{Construction}

\theoremstyle{remark}

\newtheorem{remark}[thm]{Remark}

\newtheorem{discussion}[thm]{Discussion}

\makeatletter
\def\theenumi{\@alph\c@enumi}
\makeatother

\DeclareMathOperator{\charact}{char}
\DeclareMathOperator{\homology}{H}
\DeclareMathOperator{\rhomo}{\widetilde H}

\newcommand{\naturals}{\mathbb{N}}
\newcommand{\ints}{\mathbb{Z}}
\newcommand{\rationals}{\mathbb{Q}}
\newcommand{\minus}{\ensuremath{\!\smallsetminus\!}}

\DeclareMathOperator{\ass}{Ass}
\DeclareMathOperator{\height}{ht}

\DeclareMathOperator{\tor}{Tor}
\DeclareMathOperator{\reg}{reg}

\DeclareMathOperator{\deln}{del}
\DeclareMathOperator{\link}{lk}
\DeclareMathOperator{\Star}{star}

\newcommand{\hilbFn}[2][\Bbbk]{\dim_{#1}\left[#2\right]}

\newcommand{\simplex}{\Delta}

\newcommand{\define}[1]{{\emph{#1}}}
\newcommand{\defeq}{:=}

\let\oldendremark\endremark
\def\endremark{\hfill\qedsymbol\oldendremark}
\let\oldendexample\endexample
\def\endexample{\hfill\qedsymbol\oldendexample}
\let\oldendconstruction\endconstruction
\def\endconstruction{\hfill\qedsymbol\oldendconstruction}

\title{Dependence of Betti Numbers on Characteristic}

\author{Kia Dalili}
\email{dalilik@missouri.edu}
\address{Department of Mathematics, University of Missouri, Columbia, MO}

\author{Manoj Kummini}
\email{nkummini@math.purdue.edu}
\address{Department of Mathematics, Purdue University, West Lafayette, IN}

\begin{document}

\begin{abstract}
We study the dependence of graded Betti numbers of monomial ideals on the
characteristic of the base field. The examples we describe include
bipartite ideals, Stanley--Reisner ideals of vertex-decomposable complexes
and ideals with componentwise linear resolutions. 
We give a description of bipartite graphs and, using discrete Morse theory,
provide a way of looking at the homology of arbitrary simplicial complexes
through bipartite ideals. We also prove that the Betti table of a monomial
ideal over the
field of rational numbers can be obtained from the Betti table over any
field by a sequence of consecutive cancellations.
\end{abstract}

\maketitle

\section{Introduction}

Let $R = \Bbbk[V]$ be a polynomial ring with a finite set $V$ of
indeterminates over a field $\Bbbk$. We consider $R$ to be \define{standard
graded}, \textit{i.e.}, $\deg x = 1$ for all $x \in V$. Write $\mathfrak m$
for the unique homogeneous maximal ideal $(V)R$. Let $M$ be a finitely generated
graded $R$-module. A \define{minimal graded free resolution} of $M$ is a complex
\begin{equation}
\label{equation:FFRdefn}
\xymatrix{%
F_\bullet : & 0 \ar[r] & F_p \ar[r]^{\phi_n} & \cdots
\ar[r]^{\phi_2} & F_1 \ar[r]^{\phi_1} & F_0 \ar[r] & 0
}
\end{equation}
of finitely generated graded free $R$-modules and homomorphisms such that
\begin{inparaenum}
\item for all $i \geq 1$, $\phi_i$ is of degree $0$,
\item for all $i \geq 1$, $\phi_i(F_i) \subseteq \mathfrak m F_{i-1}$, and
\item $\homology_0(F_\bullet) \simeq M$ and $\homology_i(F_\bullet) = 0$
for all $i \geq 1$. 
\end{inparaenum}
The numerical information of a free resolution, \textit{i.e.}, the degrees
of minimal generators of the $F_i$ is captured in the list of Betti numbers
of $M$; the \define{$(i,j)$th graded Betti number} of $M$, denoted
$\beta_{i,j}(M)$, is the number of minimal homogeneous generators of $F_i$ of
degree $j$. The \define{$i$th total Betti number} of $M$ is $\beta_i(M) =
\sum_j \beta_{i,j}(M)$.  
We have $\beta_{i,j}(M) = \hilbFn {\tor_i^R(\Bbbk, M)}_j$, so it
is an invariant of $M$, independent of the choice of the
free resolution $F_\bullet$. The set of graded Betti numbers is represented
in terms of a \define{Betti table} $\beta(M)$, in which the entry at column
$i$ and row $j$ is $\beta_{i, i+j}(M)$. Similarly, if $G_\bullet$ is
a complex of finitely generated graded free $R$-modules and homomorphisms,
we write $\beta(G_\bullet)$ for the Betti table of $G_\bullet$, in which the
entry at column $i$ and row $j$ is $\hilbFn {G_i \otimes_R \Bbbk}_{i+j}$.
Here we wish to understand the following question:
\begin{question}
\label{qn:conditionsForIndep}
Suppose that $I$ is a monomial $R$-ideal. Under what conditions is
$\beta(I)$ independent of the characteristic of $\Bbbk$?

\end{question}

We will see below (Proposition~\ref{propn:polarization}) that we can
immediately reduce to the case that $I$ is generated by squarefree
monomials. Then using Stanley--Reisner theory (specifically, Hochster's
formula relating Betti numbers to simplicial homology ---
see~\eqref{eqn:hochsterFormulaDirec} below) we can translate the
problem to one of determining whether certain simplicial complexes have
torsion-free homology. Therefore, in principle,
Question~\ref{qn:conditionsForIndep} has a straightforward answer; the
purpose of this note is to describe some sufficient conditions that would
guarantee the independence of $\beta(I)$ from $\charact \Bbbk$. We will
also give some examples of ideals with strong combinatorial properties,
which, nonetheless, have Betti tables that depend on $\charact \Bbbk$.

This work is motivated in part by questions raised by J.~Herzog and by the
paper of M.~Katzman~\cite{KatzmanCharIndep06}. Various authors have studied
the dependence of Betti tables on the characteristic. In~\cite
{TeraiHibiBettiNos96}, N.~Terai and T.~Hibi showed that if $I$ is generated
by quadratic square-free monomials, then $\beta_2(I)$ and $\beta_3(I)$ do
not depend on $\charact \Bbbk$. It follows from a result of B.~Xu~\cite{
XuPlanarCliques01}*{Lemma~26} that if $I$ is generated by 
quadratic square-free monomials and the $1$-skeleton of the
Stanley--Reisner complex of $I$ is a planar graph, then all the Betti
numbers of $I$ are independent of the characteristic.

We begin with describing polarization and quoting some relevant results in
combinatorial commutative algebra.
In Section~\ref{sec:whiskers}, we
will give a construction of vertex-decomposable
(Definition~\ref{defn:vertexDec}) simplicial complexes whose
Stanley--Reisner ideals have Betti tables that depend on $\charact \Bbbk$.
Section~\ref{sec:bipartiteGraphs} describes bipartite ideals. Given a
simplicial complex, we construct a bipartite ideal whose Betti numbers give
the homology of the simplicial complex, using which we exhibit an example
of a bipartite ideal $I$ such that $\beta(I)$ depends on $\charact \Bbbk$.
In Section~\ref{sec:cancellation}, we look at consecutive cancellations in
Betti tables (Definition~\ref{defn:cancellation}) and show that
ideals with componentwise linear resolution have Betti tables independent
of the characteristic. We will use~\cite{eiscommalg} as a general reference
in commutative algebra, and~\cite{BrHe:CM} and~\cite{MiStCCA05} for its
relation to combinatorics.

\section{Preliminaries}

We will use $V$ to denote an arbitrary set of vertices, as well as the
variables in the polynomial ring $R = \Bbbk[V]$. Write $V = \{x_1, \ldots,
x_n\}$. For a monomial $R$-ideal $I$, a \define{polarization} of $I$ in a
larger polynomial ring $R'$ is the squarefree monomial ideal $I'$ generated
by monomials $\prod_{i=1}^n\prod_{j=1}^{a_i} x_{i,j}$ for every minimal
monomial
generator $x_1^{a_1}\cdots x_n^{a_n}$ of $I$. For example, a polarization
of $(x_1^2, x_1x_2, x_2^3)$ is $(x_{1,1}x_{1,2}, x_{1,1}x_{2,1},
x_{2,1}x_{2,2}x_{2,3})$. See~\cite{MiStCCA05}*{Section~3.2 and
Exercise~3.15} for details. We can get a minimal free resolution of
$R/I$ from a minimal free resolution of $R'/I'$. Thus:

\begin{propn}
\label{propn:polarization}
Suppose that $I$ is a monomial $R$-ideal. Let $I'$ be a polarization of $I$
in a larger polynomial ring $R'$. Then, $\beta(I)$ depends on $\charact
\Bbbk$ if and only if $\beta(I')$ depends on $\charact \Bbbk$.
\hfill\qedsymbol
\end{propn}

\subsection*{Hochster's Formula} 
%(See~\protect{\cite{BrHe:CM}*{Equation~(10), p.~235}})
(See~\cite{MiStCCA05}*{Corollary~5.12 and Corollary~1.40}.)
For $\sigma \subseteq V$, we denote by
$\Delta|_\sigma$ the simplicial complex obtained by taking all the faces of
$\Delta$ whose vertices belong to $\sigma$. Note that $\Delta|_\sigma$ is
the Stanley-Reisner complex of the ideal $I \cap \Bbbk[\sigma]$.  First, the multidegrees $\sigma$
with $\beta_{i,\sigma}(R/I) \neq 0$ are squarefree. Secondly, for all
squarefree multidegrees $\sigma$, 
\begin{equation}
\beta_{i,\sigma} (R/I) = \dim_\Bbbk \rhomo_{\vert \sigma \vert - i -
1}(\Delta \vert_\sigma; \Bbbk)
\label{eqn:hochsterFormulaDirec}
\end{equation}

Let $I \subseteq R = \Bbbk[V]$ be a squarefree monomial ideal. Let $W
\subseteq V$ and $J = (I \cap \Bbbk[W])R$. Then,
\begin{equation}
\label{eqn:bettiNosRestr}
\beta_{i,\sigma}(R/J) = 
\begin{cases}
0, & \sigma \nsubseteq W, \\
\beta_{i,\sigma}(R/I), & \sigma \subseteq W.
\end{cases}
\end{equation}

\begin{remark}
\label{remark:depOnCharIffTorsion}
Let $\Delta$ be a simplicial complex with Stanley--Reisner ideal $I$. Then,
by~\eqref{eqn:hochsterFormulaDirec} and the universal coefficient theorem
for homology~\cite {HatcAlgTop02}*{Theorem~3A.3}, we see that
$\beta(I)$ depends on $\charact \Bbbk$ if and only if the groups
$\homology_*(\Delta; \ints)$ have torsion.
\end{remark}

\begin{example}[G.~Reisner~\cite{BrHe:CM}*{Section~5.3}]
\label{example:Reisner}
Let $\Delta$ be the minimal triangulation of $\mathbb{RP}^2$ on the vertex
set $V = \{x_1, \ldots, x_6\}$ with facets
$x_4x_5x_6$, $x_3x_5x_6$, $x_2x_4x_6$, $x_1x_3x_6$, $x_1x_2x_6$,
$x_1x_4x_5$, $x_2x_3x_5$, $x_1x_2x_5$, $x_2x_3x_4$ and $x_1x_3x_4$.
Then
$I = (x_1x_2x_3, x_1x_2x_4, x_1x_3x_5, x_2x_4x_5,
x_3x_4x_5, x_2x_3x_6, x_1x_4x_6, x_3x_4x_6, x_1x_5x_6, x_2x_5x_6)$.
The Betti table of $I$ depends on $\charact \Bbbk$, owing to the fact that
$\rhomo_1(\mathbb{RP}^2;\ints) \simeq \ints/2$.  
When $\charact \Bbbk =2$ and when $\charact \Bbbk \neq
2$, $\beta(I)$ is, respectively:

\begin{center}
\begin{tabular}{lcr}
\begin{tabular}{r|ccccc}
\hline
          & 0& 1& 2&3&4 \\
\hline
     total& 1&10&15&7&1 \\
\hline
         0& 1& .& .&.&. \\
         1& .& .& .&.&. \\
         2& .&10&15&6&1 \\
         3& .& .& .&1&. \\
\hline
\end{tabular}
& \hspace{3em} or \hspace{3em} &
\begin{tabular}{r|cccc }
\hline
          & 0& 1& 2&3 \\
\hline
     total& 1&10&15&6 \\
\hline
         0& 1& .& .&. \\
         1& .& .& .&. \\
         2& .&10&15&6 \\
\hline
\end{tabular}
\end{tabular}
\end{center}
\end{example}

\begin{remark}
Let $\Delta$ be any simplicial complex on $V$ and $x \in V$.
Then there exists a decomposition $\Delta = \Star_{\Delta}(x) \cup
\deln_{\Delta}(x)$, where $\Star_{\Delta}(x) = \{F \in \Delta : F \cup
\{x\} \in \Delta\}$ and $\deln_{\Delta}(x) = \Delta|_{V \minus \{x\}}$.
Note that $\Star_{\Delta}(x) \cap \deln_{\Delta}(x) =
\link_{\Delta}(x)$, called the \define{link} of $x$ in $\Delta$.
Its Stanley-Reisner ideal in $\Bbbk[V \minus \{x\}]$ is
$(I: x) \cap \Bbbk[V \minus \{x\}]$.
\end{remark}

\begin{discussion}
\label{discussion:MV}
Let $\Delta$ be any simplicial complex on $V$ and $x \in V$. Since $\Star_{\Delta}(x)$ is a cone over $x$, we obtain, from the
Mayer--Vietoris sequence on homology~\cite{HatcAlgTop02}*{Section~2.2}, the
following exact sequence:
\begin{equation}
\xymatrix{%
& \cdots \ar[r] &
\rhomo_i(\link_{\Delta}(x); \ints) \ar[r]  &
\rhomo_i(\deln_{\Delta}(x); \ints) \ar[r] &
\rhomo_i(\Delta; \ints) \ar[r] & \cdots \ar[r] &
\rhomo_0(\Delta; \ints) \ar[r]  &0.
}
\end{equation}
In particular, if $\rhomo_*(\deln_{\Delta}(x); \ints) = 0$, then 
$\rhomo_{i+1}(\Delta; \ints) \simeq \rhomo_i(\link_{\Delta}(x); \ints)$,
for all $i \geq 0$. \hfill\qedsymbol
\end{discussion}

\section{Ideals containing powers}
\label{sec:whiskers}

Let $I$ be a monomial $R$-ideal containing $x_1^i$ for some $i \geq 1$. In
Theorem~\ref{thm:powers} we describe when $\beta(I)$ would be independent
of the characteristic, from which we derive a result of M.~Katzman and
construct examples of vertex-decomposable simplicial complexes whose free
resolution depends on the characteristic.

\begin{thm}
\label{thm:powers}
Let $I$ be a monomial $R$-ideal containing $x_1^i$ for some $i \geq 1$.
Write $I = (J, x_1^t)$ minimally, \textit{i.e.}, $t$ is the least integer
such that $x_1^t \in I$ and $J$ is generated by the elements of $I$ not
divisible by $x_1^t$. Then the following are
equivalent:
\begin{enumerate}
\item $\beta(I)$ is independent of $\charact \Bbbk$.
\item Both $\beta(J)$ and $\beta((I\!:_R\!x_1))$ are independent of
$\charact \Bbbk$.
\end{enumerate}
\end{thm}

\begin{proof}
If $t=1$, then $I = (J, x_1)$ and $J$ is an ideal extended from $\Bbbk[x_2,
\ldots, x_n]$. Since $x_1$ is a nonzerodivisor on $R/J$, we see that
$\beta(I)$ depends on the characteristic if and only if $\beta(J)$ depends
on the characteristic. Therefore we may assume that $t \geq 2$.

We will use polarization (by Proposition~\ref{propn:polarization}) to
reduce to the case of squarefree monomial ideals. Let $I'$ be a
polarization of $I$ in a polynomial ring $R'$; we will denote the variables
that correspond to $x_1$ by $y_1, \ldots, y_t$ and those that correspond to
$x_2, \ldots, x_n$ by $z_1, \ldots, z_m$. Write $I'$ minimally as $(J',
y_1\cdots y_t)$. Note that $J'$ and $(I'\!:_{R'}\!y_1)$ are, respectively,
the polarization of $J$ and $(I\!:_R\!x_1)$ in $R'$. By
Proposition~\ref{propn:polarization}, it suffices to show that $\beta(I')$
is independent of $\charact \Bbbk$ if and only if both $\beta(J')$ and
$\beta((I'\!:_{R'}\!y_1))$ are independent of $\charact \Bbbk$.  Since $J'
= (I' \cap \Bbbk[y_1, \ldots, y_{t-1}, z_1, \ldots, z_m])R'$, we see,
by~\eqref{eqn:bettiNosRestr}, that if $\beta(I')$ is independent of
$\charact \Bbbk$ then $\beta(J')$ is independent of $\charact \Bbbk$.
Therefore, we will assume that $\beta(J')$ is independent of $\charact
\Bbbk$ and show that $\beta(I')$ is independent of $\charact \Bbbk$ if and
only if $\beta((I'\!:_{R'}\!y_1))$ is independent of $\charact \Bbbk$.

Suppose that $\beta_{i,\tau}(I')$ depends on $\charact \Bbbk$, for some
$\tau \subseteq \{y_1, \ldots, y_t, z_1, \ldots, z_m\}$ and $i$. Then 
$\{y_1, \ldots, y_t\} \subseteq \tau$, for, otherwise,
$\beta_{i,\tau}(I') = \beta_{i,\tau}(J')$. Let $\Delta$ be the
Stanley--Reisner complex of $I' \cap \Bbbk[\tau]$ on the vertex set
$\tau$. Since every generator of $I'$ that is divisible by $y_2$ is also
divisible by $y_1$, we see that $\deln_{\Delta}(y_1)$ is a cone over $y_2$;
in fact, it is a cone over the simplex on $y_2, \ldots, y_t$. On the other
hand, $\link_{\Delta}(y_1)$ is the Stanley--Reisner complex (on $\tau
\minus \{y_1\}$) of $(I'\!:_{R'}\!y_1) \cap \Bbbk[\tau \minus \{y_1\}]$. By
Discussion~\ref{discussion:MV}, Remark~\ref{remark:depOnCharIffTorsion}
and~\eqref{eqn:bettiNosRestr}, we see that $\beta((I'\!:_R\!y_1))$
depends on $\charact \Bbbk$. 

Conversely, assume that $\beta_{i, \sigma}((I'\!:_R\!y_1))$ depends on
$\charact \Bbbk$ for some $\sigma \subseteq \{y_1, \ldots, y_t, z_1,
\ldots, z_m\}$ and $i$. Then $y_1 \not \in \sigma$. Write $\tau = \sigma
\cup \{y_1\}$. Now, reversing the above argument, we see that $\beta(I')$
depends on $\charact \Bbbk$.
\end{proof}

\begin{cor}[Katzman~\protect{\cite{KatzmanCharIndep06}*{Corollary~1.6}}]
\label{thm:katzmanWhiskers}
Let $I$ be quadratic squarefree monomial $R$-ideal, and let $y$ be
algebraically independent over $R$. Then $\beta((IR[y], x_1y))$ is
independent of $\charact \Bbbk$ if and only if $\beta(I)$ is independent of
$\charact \Bbbk$.
\end{cor}

\begin{proof}
It suffices to show that if $\beta(I)$ is independent of $\charact \Bbbk$,
then $\beta((I\!:_R\!x_1))$ is independent of $\charact \Bbbk$.
Let $\sigma = \{x_i : x_1x_i \not \in I\}$.  Then $(I\!:_R\!x_1) = (I \cap
\Bbbk[\sigma])R + (\sigma)R$.  If $\beta(I)$ is independent of $\charact
\Bbbk$, then $\beta((I \cap \Bbbk[\sigma])R)$, and, hence,
$\beta((I\!:_R\!x_1))$ are independent of $\charact \Bbbk$.
\end{proof}

%\begin{defn}[\protect{\cite{BjornerWachsShellableII97}*{Definition~11.1}}]
\begin{defn}[\protect{\cite{PrBiDecSimpCx80}*{Definition~2.1}}]
\label{defn:vertexDec}
Let $\Delta$ be a $d$-dimensional simplicial complex on a vertex set
$V$. We say that $\Delta$ is \define{vertex-decomposable} if it is
pure-dimensional and either $\Delta$ is the $d$-simplex, or there exists $x
\in V$ such that
\begin{inparaenum}
\item $\link_\Delta(x)$ is $(d-1)$-dimensional and vertex-decomposable, and
\item $\deln_{\Delta}(x)$ is $d$-dimensional and vertex-decomposable.
\end{inparaenum}
\end{defn}

Note that $\link_\Delta(x)$ is $(d-1)$-dimensional and vertex-decomposable
if and only if $\Star_{\Delta}(x)$ is $d$-dimensional and
vertex-decomposable. If $\Delta$ is vertex-decomposable, then it is
shellable and, hence, Cohen-Macaulay in all characteristics.

We say that a $R$-ideal $I$ is \define{primary} if $R/I$ has a unique
associated prime. A monomial $R$-ideal $I$ is primary (with associated
prime ideal $\mathfrak p$) if and only if $\mathfrak p$ is the radical of
$I$ and no minimal monomial generator of
$I$ is divisible by a variable not in $\mathfrak p$. (Note that $\mathfrak
p$ is generated by a subset of the variables.)

\begin{propn}
\label{propn:polVertDec}
Stanley--Reisner complexes of the polarization of primary monomial ideals
are vertex-decomposable.
\end{propn}

\begin{proof}
Let $\mathcal S$ be the set of simplicial complexes on a vertex set $V$.
This is a poset, under inclusion: $\Delta' \subseteq \Delta$ if $F \in
\Delta$ for every $F \in \Delta'$. By induction on $\mathcal S$, it
suffices to show that if $\Delta$ is the Stanley--Reisner complex of the
polarization of a primary ideal, then there exists $x \in V$ such that the
Stanley--Reisner ideals of $\Star_{\Delta}(x)$ and $\deln_{\Delta}(x)$ are
also obtained through polarization.

Let $1 \leq c \leq |V|$, and $I$ a squarefree monomial ideal with $\height
I = c$. Then $I$ is the polarization of a primary monomial ideal if and
only if there exists a partition $V = \bigsqcup_{i=1}^c \{x_{i,1}, \ldots,
x_{i, n_i}\}$ of the vertex set such that for every $1 \leq i \leq c$ and
for every generator $f$ of $I$, if $x_{i,j} \mid f$ for some $1 \leq j \leq
n_i$, then $x_{i,k} \mid f$ for every $1 \leq k \leq j$. Moreover, if this
holds, we may assume that $I$ is the polarization of an monomial ideal
primary to $(x_{1,1}, \ldots, x_{c,1})R$.

Let $\mathfrak a$ be an $(x_{1,1}, \ldots, x_{c,1})$-primary monomial ideal
and $I$ its polarization. Let $\Delta$ be the Stanley--Reisner complex of
$I$. The Stanley--Reisner ideal of $\Star_{\Delta}(x_{1,1})$ is
$(I\!:\!x_{1,1})$, which is a polarization of $(\mathfrak a\!:\!x_{1,1})$.
The Stanley--Reisner ideal of $\deln_{\Delta}(x_{1,1})$ is $(I,x_{1,1})$,
which is a polarization of $(\mathfrak a,x_{1,1})$. Both 
$(\mathfrak a\!:\!x_{1,1})$ and 
$(\mathfrak a,x_{1,1})$ are primary.
\end{proof}

\begin{remark}
\label{remark:vertDecExample}
We now see that vertex-decomposability does not ensure that Betti tables
are independent of $\charact \Bbbk$. For, let $I$ be as in
Example~\ref{example:Reisner}. Let $S = R[y_1, \ldots, y_n]$. Let $J = IS +
(x_1y_1, \ldots, x_ny_n)$; it is the polarization of $I + (x_1^2, \ldots,
x_n^2)$ which is $(x_1, \ldots, x_n)$-primary. Therefore $\Delta_J$ is
vertex-decomposable, while 
$\beta(J)$ depends on $\charact \Bbbk$, 
by Theorem~\ref{thm:powers} and
Proposition~\ref{propn:polarization}. This behaviour is already known for
shellable
complexes~\cite{TeraiHibiBettiNos96}*{Examples~3.3, 3.4}.
\end{remark}

\section{Bipartite Ideals}
\label{sec:bipartiteGraphs}

We say that a quadratic
monomial ideal $I$ is \emph{bipartite} if there exists a partition $V =
V_1 \sqcup V_2$ such that every minimal generator of $I$ is of the form
$xy$ for some $x \in V_1$ and $y \in V_2$.
Construction~\ref{constr:bipFromCx} describes all bipartite ideals. 
In Theorem~\ref{thm:homologyIsom}, we 
give a method to calculate the homology of arbitrary simplicial complexes,
similar to the method of nerve complexes.

\begin{constr}
\label{constr:coning}
Let $\Gamma$ be a simplicial complex on $V_1 \defeq \{x_1, \ldots,
x_n\}$. Let $\Gamma_j, 1 \leq j \leq m$ be a collection of simplicial
subcomplexes of $\Gamma$ such that $\Gamma = \cup_{j=1}^m \Gamma_j$.
Let $V_2 = \{y_1, \ldots, y_m\}$ be a set of $m$ new vertices.
Define 
\begin{equation}
%\label{equation:
\widetilde{\Gamma} = \left\{ \sigma \cup \tau : \sigma \in \Gamma, \tau \subseteq
\left\{y_j : \sigma \in \Gamma_j\right\}
\right\}.
\end{equation}
\end{constr}

\begin{lemma}
\label{lemma:contractible}
With notation as above, $\widetilde{\Gamma}$ is contractible.
\end{lemma}

\begin{proof}
%\label{proof:
We prove this using discrete Morse theory developed by
R.~Forman~\cites{FormanDMT98}. Refer to the exposition in
\cite{FormanUsersGuideDMT02} for unexplained terminology.  Specifically, we
will exhibit a complete acyclic matching on the Hasse diagram of $\widetilde{\Gamma}$;
see~\cite{ChariDMT00}*{Section~3}
and~\cite{FormanUsersGuideDMT02}*{Section~6} for the interpretation of
acyclic matchings of the Hasse diagram in terms of discrete Morse theory.

Let $\sigma \in \Gamma$. Let $Y_\sigma = \{y_j : \sigma \in \Gamma_j\}$ and
$\mathcal F_\sigma = \{ \sigma \cup \tau : \tau \subseteq Y_\sigma\}$. Then
$\widetilde{\Gamma} = \bigsqcup_{\sigma \in \Gamma} \mathcal F_\sigma$ is a
partition. Let $j$ be the smallest integer such that $y_j \in Y_\sigma$. We
define a complete matching on $\mathcal F_\sigma$ by connecting $\sigma
\cup \tau$ with $\sigma \cup \tau \cup \{y_j\}$ for all $\tau \subseteq
Y_\sigma$ with $y_j \not \in \tau$. Repeating this for all $\sigma \in
\Gamma$, we obtain a complete matching of the Hasse diagram of
$\widetilde{\Gamma}$. We now claim that this is an acyclic matching. Assume
the claim; then $\widetilde{\Gamma}$ is contractible,
by~\cite{FormanUsersGuideDMT02}*{Theorem~6.4}.

To prove the claim, we let, for a face $F$ of $\widetilde{\Gamma}$, 
\[
j_F = 
\begin{cases}
%\label{cases:
\min \{j : y_j \in F\}, & \text{if there exists}\; j \;\text{such that}\;
y_j \in F \\
\infty, & \text{otherwise}.
\end{cases}
\]
Let $F \rightarrow F' \rightarrow F''$ be edges 
in the Hasse diagram (modified, as in~\cite{
FormanUsersGuideDMT02}*{Section~6}, to include the matchings),
such that one of them is an
up arrow and the other is a down arrow. Then $j_F > j_{F''}$.
Since every edge in the Hasse diagram connects two faces whose sizes differ
exactly by one, we see that every cycle has an even number of edges. Since
no two up arrows share a vertex (the up arrows form the matching), the up
and the down arrows alternate in every directed cycle. Hence the Hasse
diagram does not have directed cycles.
\end{proof}

\begin{remark}
Note that there may exists $j$ such that $\Gamma_j = \{\varnothing\}$.
\end{remark}

\begin{constr}
\label{constr:bipFromCx}
Let $\Gamma$ be a simplicial complex on $V_1 \defeq \{x_1, \ldots,
x_n\}$. Denote the number of facets of $\Gamma$ by $m$. Let $G_j, 1 \leq j
\leq m$ be such
that for all $1 \leq j \leq m$, $V_1 \minus G_j$ is a face of $\Gamma$ and
such that every facet of $\Gamma$ is of the form $V_1 \minus G_j$ for some
$j$. Let $y_1, \ldots, y_m$ be new vertices. Let $\simplex_{V_1}$
be the $(n-1)$-simplex on $x_1, \ldots, x_n$. Define 
\begin{equation}
%\label{equation:
\Delta' = \left\{ \sigma \cup \tau : \sigma \in \Gamma, \tau \subseteq
\left\{y_j : \sigma \subseteq \left(V_1 \minus G_j\right) \right\}
\right\} \qquad \text{and}\qquad
\Delta = \Delta' \bigcup \simplex_{V_1}.
\end{equation}
Let $I$ be the Stanley--Reisner ideal of $\Delta$, in the ring $R =
\Bbbk[x_1, \ldots, x_n, y_1, \ldots, y_m]$. Let $I_\Gamma$ denote the
extension of the Stanley--Reisner ideal of $\Gamma$ from the ring
$\Bbbk[x_1, \ldots, x_n]$ to $R$.
\end{constr}

\begin{propn}
\label{propn:primaryDec}
With notation as above, 
$I = \left(x_iy_j : 1 \leq j \leq m, x_i \in G_j\right)$. Moreover,
$I = (I + I_\Gamma) \cap (y_1, \ldots, y_m)$. Hence the Stanley--Reisner
ideal of $\Delta'$ is $(I + I_\Gamma)$.
\end{propn}

\begin{proof}
%\label{proof:
We will first show that the minimal nonfaces of $\Delta$ are precisely
$\{x_i,y_j\}, 1 \leq j \leq m, x_i \in G_j$.  It follows from the
definition of $\Delta$ that for every $1 \leq j \leq m$ and $x_i \in G_j$, 
$\{x_i,y_j\}$ is a nonface.
Observe that $\{x_1, \ldots, x_n\}$ and $\{y_1, \ldots, y_m\}$ are faces of
$\Delta$. Let $\sigma \cup \tau$ with $\sigma \subseteq \{x_1, \ldots,
x_n\}$ and $\tau \subseteq \{y_1, \ldots, y_m\}$ be a minimal nonface of
$\Delta$. Hence $\sigma \neq \varnothing \neq \tau$. Therefore there exists
$y_j \in \tau$ such that $\sigma \not \in V_1 \minus G_j$. Let $x_i
\in \sigma \cap G_j$. Now, $\{x_i, y_j\} \subseteq \sigma \cup \tau$; by
minimality of $\sigma \cup \tau$ we conclude that 
$\sigma \cup \tau = \{x_i, y_j\}$. 

In order to prove that 
$I = (I + I_\Gamma) \cap (y_1, \ldots, y_m)$, it suffices to show that $f
\in I$ for all monomials $f \in I_\Gamma \cap (y_1, \ldots, y_m)$. Since
the generators of $I_\Gamma$ are monomials in $V_1$, write $f =
f'y_j$ for some $f' \in I_\Gamma$. Let $f'$ correspond to a nonface
$\sigma$ of $\Gamma$. Therefore $\sigma \cup\{y_j\}$ is a nonface of
$\Delta$, so $f \in I$.

Note that $(I + I_\Gamma) \nsubseteq (y_1, \ldots, y_m)$. Hence the
intersection $(I + I_\Gamma) \cap (y_1, \ldots, y_m)$ corresponds to the
union $\Delta' \cup \simplex_{V_1}$;
see~\cite {MiStCCA05}*{Theorem~1.7}.
Therefore the Stanley--Reisner ideal of
$\Delta'$ is $(I + I_\Gamma)$.
\end{proof}

\begin{remark}
%\label{remark:
Every bipartite $R$-ideal $I$, with the partition $\{x_1, \ldots, x_n\}
\sqcup \{y_1, \ldots y_m\}$, arises through
Construction~\ref{constr:bipFromCx}. Write $I = \left(x_iy_j : 1 \leq j
\leq m, x_i \in G_j\right)$, where the $G_j$ are subsets of $V_1$.
Let
\[
J = \left(\prod_{x_i \in F} x_i : F \cap G_j \neq \varnothing
\;\text{for all}\; j\right) = \left(\prod_{x_i \in F} x_i : F \cap G_j \neq
\varnothing \;\text{for all minimal}\; G_j \right).
\]
Then $I = (I+J) \cap (y_1, \ldots y_m)$. Let $\Gamma$ be the
Stanley--Reisner complex of $J$ on the vertex set $V_1$. 
The facets of $\Gamma$ are $V_1 \minus G_j$ for $G_j$ minimal. To see
this, it suffices to show that 
\[
\ass(R/J) = \{(G_j)R : G_j \;\text{minimal}\}
\]
or, equivalently, that 
\[
J  = \left(\prod_{\substack{x \in G_j \\ G_j \;\text{minimal}}}
x\right)^\vee \qquad\qquad (\text{here}\; (-)^\vee \;\text{denotes taking the
Alexander dual})
\]
which follows from the definition of $J$
and~\cite{FariFacetIdeal02}*{Proposition~1}. 
Now apply Construction~\ref{constr:bipFromCx}
with the $G_j$ as above.
\end{remark}

\begin{thm}
\label{thm:homologyIsom}
Let $\Gamma$ and $\Delta$ be as in Construction~\ref{constr:bipFromCx}.
Then for all $i \geq 0$, $\rhomo_{i+1}(\Delta; \ints) \simeq 
\rhomo_i(\Gamma; \ints)$.
\end{thm}

\begin{proof}
%\label{proof:
Notice that $\Delta' \cap \Delta_{V_1} = \Gamma$, or, equivalently, that
$(I + I_\Gamma) + (y_1, \ldots, y_m) = I_\Gamma + (y_1, \ldots, y_m)$.
From the Mayer--Vietoris sequence on
homology~\cite{HatcAlgTop02}*{Section~2.2},
it suffices to prove that $\rhomo_i(\Delta';\ints) = 0$ for all $i \geq 0$.
This follows from Lemma~\ref{lemma:contractible}.
\end{proof}

J.~Herzog raised the question whether the Betti tables of bipartite
ideals are independent of the characteristic. 

\begin{example}
Let $R = \ints[x_1, \ldots, x_6, y_1, \ldots, y_{10}]$. Let $\Gamma$ be the
minimal triangulation of $\mathbb{RP}^2$ on the vertices $x_1, \ldots,
x_6$, given in Example~\ref{example:Reisner} (and called $\Delta$ there).
By Theorem~\ref{thm:homologyIsom}, $\rhomo_*(\Delta; \ints)$ is not
torsion-free so, $\beta(I)$ depends on $\charact \Bbbk$.
\end{example}

\section{Componentwise linear resolutions}
\label{sec:cancellation}

We look at consecutive cancellation in Betti tables, and use it to show
that the Betti tables of ideals with componentwise linear resolution are
independent of the characteristic. For $t \in \naturals$, we write $(I_t)R$
for the ideal generated by the vector space $I_t$ of polynomials of degree
$t$ in $I$. We say that the resolution of $I$ is \define{$t$-linear} if $I
= (I_t)R$ and $\beta_{i,j}(I) = 0$ for all $j \neq i + t$ and for all $i$.
We say that an $R$-ideal $I$ has a \define{componentwise linear resolution}
(see~\cite{HeHiCptLin99}) if, for all $t \in
\naturals$, the resolution of $(I_t)R$ is $t$-linear.
\begin{thm}
\label{thm:cpntLinear}
Suppose that $I$ is a monomial $R$-ideal that has a componentwise linear
resolution, in all characteristics. Then $\beta(I)$ does not
depend on $\charact \Bbbk$.
\end{thm}

\begin{defn}[\cite{PeevConsecCanc04}]
\label{defn:cancellation}
Let $\beta$ and $\beta'$ be Betti tables. We say that $\beta'$ is obtained
from $\beta$ by a \define{consecutive cancellation} if there exists $i,j$
such that $\beta'_{i,j} = \beta_{i,j}-1$, $\beta'_{i+1,j} =
\beta_{i+1,j}-1$ and $\beta'_{k,l} = \beta_{k,l}$ if $(k,l) \neq (i,j)$ and 
$(k,l) \neq (i+1,j)$.
\end{defn}

For instance, in Example~\ref{example:Reisner}, the Betti table of $R/I$ in
characteristic $0$ is obtained from its Betti table in characteristic $2$
by a consecutive cancellation; we have $i=3$ and $j=6$.

\begin{propn}
\label{thm:flatnessAndIndependence}
Let $A = \ints[x_1, \ldots, x_n]$.
Let $\mathfrak a$ be a
homogeneous $A$-ideal such that every integer is a nonzerodivisor on
$A/\mathfrak a$. Then, for all primes $p$, $\beta^{A\otimes_\ints
\rationals}\left( (A/\mathfrak a) \otimes_\ints \rationals \right)$ can be
obtained from $\beta^{A/pA}\left( (A/\mathfrak a) \otimes_\ints
(\ints/p\ints) \right)$ by a sequence of consecutive cancellations.
\end{propn}

\begin{proof}
Note that $A/\mathfrak a$ is a flat $\ints$-algebra.
Let $\mathbb F_\bullet$ be a minimal graded free $A \otimes_\ints
\ints_{(p)}$-resolution of $(A/\mathbf a) \otimes_\ints \ints_{(p)}$. Then
$\beta^{A/pA}\left(
(A/\mathfrak a) \otimes_\ints (\ints/p\ints) \right) = \beta(\mathbb
F_\bullet)$. Now, $\mathbb F_\bullet \otimes_{\ints_{(p)}} \rationals$ is a
graded free $(A\otimes_\ints \rationals)$-resolution of 
$(A/\mathfrak a) \otimes_\ints \rationals$.
Therefore we can write $\mathbb F_\bullet
\otimes_{\ints_{(p)}} \rationals = \mathbb G_\bullet \oplus \mathbb
G'_\bullet$ where $\mathbb G_\bullet$ is a minimal graded free
$(A\otimes_\ints \rationals)$-resolution of $(A/\mathfrak a) \otimes_\ints
\rationals$ and $\mathbb G'_\bullet$ is graded trivial complex of free
$(A\otimes_\ints \rationals)$-modules~\cite{eiscommalg}*{Theorem~20.2}.
Therefore $\beta (\mathbb G_\bullet)$ can be obtained from $\beta (\mathbb
F_\bullet)$ by a sequence of consecutive cancellations; now, note that
$\beta^{A\otimes_\ints \rationals}\left( (A/\mathfrak a) \otimes_\ints
\rationals \right) = \beta(\mathbb G_\bullet)$.
\end{proof}

The following is an elaboration of the `truncation principle' of
D.~Eisenbud, C.~Huneke and B.~Ulrich~\cite
{EHUregtor04}*{Proposition~1.6}.

\begin{lemma}
\label{lemma:trunc}
Let $t \in \naturals$. Then for all $i \geq 0$ and for all $j > i+t$,
$\beta_{i,j}(I \cap \mathfrak m^t) = \beta_{i,j}(I)$ 
\end{lemma}

\begin{proof}
The lemma follows by repeatedly applying (finitely many times) the following.
\underline{Claim}: Suppose that $I$ is minimally generated by $f_1, \ldots,
f_r$. Write $\tilde I = (f_2, \ldots, f_r) + f_1\mathfrak m$. Then
$\beta_{i,j}(\tilde I) = \beta_{i,j}(I)$ for all $i \geq 0$ and for all $j
> i+ \deg f_1+1$. To prove the claim, consider the exact sequence
\[
\xymatrix{
0 \ar[r] & \frac R{(\tilde I :_R f_1)}(-\deg f_1) \ar[r] & R/\tilde I
\ar[r] & R/I \ar[r] & 0,
}
\]
and the associated exact sequence of $\tor$, 
\[
\xymatrix{
\ar[r] & 
\tor_i(\Bbbk, \Bbbk(-\deg f_1))_j \ar[r] & 
\tor_i(\Bbbk, R/\tilde I)_j \ar[r] & 
\tor_i(\Bbbk, R/I)_j \ar[r] & 
\tor_{i-1}(\Bbbk, \Bbbk(-\deg f_1))_j \ar[r] &. 
}
\]
(Here, we use the fact that $(\tilde I :_R f_1) = \mathfrak m$.) Now, 
$\beta_{i,j}(\Bbbk(-\deg f_1))  = 0 =\beta_{i,j}(\Bbbk(-\deg f_1))$
for all $i \geq 0$ and for all $j > i+ \deg f_1$, which proves the claim.
\end{proof}

For a homogeneous $R$-ideal $I$, set $d(I)$ to be the least degree of a
minimal generator of $I$, \textit{i.e.}, $d(I) = \min \{ j : \beta_{0,j}(I)
\neq 0\}$.

\begin{propn}
\label{thm:keyArgStrands}
Let $\mathcal C$ be a class of monomial $R$-ideals such that for all $I
\in \mathcal C$,
\begin{inparaenum}
\item \label{item:linStrIndep} $\beta_{i,i+d(I)}(I)$ is independent of
$\charact \Bbbk$, and 
\item \label{item:closedInters} $I \cap \mathfrak m^{d(I)+1} \in \mathcal C$.
\end{inparaenum}
Then for all $I \in \mathcal C$, $\beta(I)$ is independent of $\charact
\Bbbk$.
\end{propn}

\begin{proof}
%\label{proof:
Let $I \in \mathcal C$.  We prove the theorem by induction on $\reg I -
d(I)$. If $\reg I = d(I)$, then the resolution of $I$ is $d(I)$-linear.
The only non-zero entries in $\beta(I)$ are $\beta_{i,i+d(I)}(I), i \geq
0$. Hence, by hypothesis~\eqref{item:linStrIndep}, $\beta(I)$ is
independent of $\charact \Bbbk$.

If $\reg I > d(I)$, then, by \eqref{item:closedInters} and 
the induction hypothesis, $\beta((I \cap
\mathfrak m^{d(I)+1})$ is independent of $\charact \Bbbk$. By
Lemma~\ref {lemma:trunc}, $\beta_{i,j}(I)$ is independent of $\charact
\Bbbk$ for all $i \geq0$ and for all $j \geq i+d(I)+2$.
Proposition~\ref{thm:flatnessAndIndependence}, along
with~\eqref{item:linStrIndep}, now finishes the proof.
\end{proof}

\begin{lemma}
\label{lemma:tLinearStr}
For all $i \geq 0$, $\beta_{i,i+d(I)}((I_{d(I)})R) =
\beta_{i,i+d(I)}(I)$.
\end{lemma}

\begin{proof}
%\label{proof:
Let $J \subseteq I$ be the subideal generated by the minimal generators of
$I$ of degree $d(I)+1$ or greater. Then $I = I_{d(I)} + J$.
Consider the exact sequence
\[
0 \rightarrow R/(I_{d(I)} \cap J) \rightarrow R/I_{d(I)} \oplus R/J
\rightarrow R/I \rightarrow 0
\]
and the associated exact sequence of $\tor$, 
\[
\xymatrix{
\ar[r] & 
\tor_i(\Bbbk, \frac R{I_{d(I)} \cap J})_j \ar[r] &
{\begin{matrix}
\tor_i(\Bbbk, R/I_{d(I)})_j \\ \oplus \\ \tor_i(\Bbbk, R/J)_j
\end{matrix}}
\ar[r] & 
\tor_i(\Bbbk, R/I)_j \ar[r] & 
\tor_{i-1}(\Bbbk, \frac R{I_{d(I)} \cap J})_j \ar[r] &.
}
\]
Now, for all $i \geq 1$,
$\beta_{i, i+d(I)}(\frac R{I_{d(I)} \cap J}) = 
\beta_{i-1, i+d(I)}(\frac R{I_{d(I)} \cap J}) = 
\beta_{i,i+d(I)}(R/J) = \beta_{i-1,i+d(I)}(R/J) = 0$. This proves the
lemma.
\end{proof}

\begin{proof}[Proof of Theorem~\ref{thm:cpntLinear}]
We will verify that ideals with componentwise linear resolution satisfy the
hypotheses of Proposition~\ref {thm:keyArgStrands}. By definition,
$(I_{d(I)})R$ has a $d(I)$-linear resolution in all characteristics. By
Proposition~\ref{thm:flatnessAndIndependence}, $\beta((I_{d(I)})R)$ does
not depend on $\charact \Bbbk$, so, by Lemma~\ref{lemma:tLinearStr}, we see
that hypothesis~\eqref{item:linStrIndep} is satisfied.
Hypothesis~\eqref{item:closedInters} is obtained from noting that
for all $t \geq d(I)+1$, $I_t = (I \cap \mathfrak m^{d(I)+1})_t$.
\end{proof}

\begin{remark}
%\label{remark:
We note that the proofs of Proposition~\ref{thm:keyArgStrands} and
Theorem~\ref{thm:cpntLinear} will hold, \textit{mutatis mutandis}, if we
replace the phrase ``$I$ is a monomial $R$-ideal'' with the phrase ``$I$ is
the image in $R$ of a $\ints[x_1, \ldots x_n]$-ideal $\mathfrak a$ such
that $\ints[x_1, \ldots x_n]/\mathfrak a$ is a flat $\ints$-algebra''.
\end{remark}

\subsection*{Examples}

Theorem~\ref{thm:cpntLinear} shows that we cannot detect dependence on the
characteristic using Alexander duality. For, let $I$ be an ideal (such as
the one in Remark~\ref{remark:vertDecExample}) such that
$R/I$ is Cohen--Macaulay in all characteristics, but $\beta(I)$ depends on
the characteristic. By a result of J.~Eagon and V.~Reiner~\cite
{MiStCCA05}*{Theorem~5.56}, its Alexander dual $I^\vee$ has a linear
resolution in all characteristics. Hence $\beta(I^\vee)$ is independent
of $\charact \Bbbk$.

On the other hand, stable ideals have componentwise
linear resolutions, given by S.~Eliahou and M.~Kervaire; see~\cite
{MiStCCA05}*{Section~2.3} and~\cite {HeHiCptLin99}*{Example~1.1}. Therefore
for any stable ideal $I$,
$\beta(I)$ is independent of $\charact \Bbbk$. 

Now, as an application of Proposition~\ref{thm:flatnessAndIndependence}, we
obtain that if $I$ is the edge ideal of a chordal graph $G$, then
$\beta(I)$ does not depend on characteristic. T.~Hibi, K.~Kimura and
S.~Murai~\cite{HiKiMuChordalBetti09}*{Theorem~2.1} show that the sequence
$(\beta_i(R/I))$ of total Betti numbers depend only on $I$. By
Proposition~\ref{thm:flatnessAndIndependence}, $\beta(I)$ is independent of
$\charact \Bbbk$. As another corollary, we see that if $R/I$ has a pure
resolution in all characteristics, then $\beta(I)$ does not depend on the
characteristic.

\section*{Acknowledgements}
We thank J.~Herzog for helpful comments. Parts of this work were
completed at the Pan American Scientific Institute Summer School on
``Commutative Algebra and its Connections to Geometry'' in Olinda, Brazil,
and when the second author visited the University of Missouri; we thank
both institutions for their hospitality. The computer algebra system
\texttt{Macaulay2} provided valuable assistance in studying examples.

%\bibliography{kummini}

%\def\cfudot#1{\ifmmode\setbox7\hbox{$\accent"5E#1$}\else
%  \setbox7\hbox{\accent"5E#1}\penalty 10000\relax\fi\raise 1\ht7
%  \hbox{\raise.1ex\hbox to 1\wd7{\hss.\hss}}\penalty 10000 \hskip-1\wd7\penalty
%  10000\box7}

\end{document}